\documentclass[12pt]{amsart}
\usepackage[top=30truemm,bottom=30truemm,left=25truemm,right=25truemm]{geometry}
\usepackage{txfonts}
\usepackage{mathrsfs}

\usepackage{color}
\usepackage{bm}
\usepackage{amsfonts,amssymb}
\usepackage{dsfont}
\usepackage{amscd}
\usepackage{extarrows}
\usepackage{amsmath}
\usepackage{mathrsfs}
\usepackage{enumerate}
\usepackage{amscd}
\usepackage[all]{xy}
\usepackage[pagebackref,colorlinks]{hyperref}

\usepackage{extarrows}
\theoremstyle{plain} 
\newtheorem{theorem}{\indent\bf Theorem}[section]
\newtheorem{lemma}[theorem]{\indent\bf Lemma}

\newtheorem{claim}[theorem]{\indent\bf Claim}

\theoremstyle{definition} 

\newtheorem{prob}[theorem]{\indent\bf Problem}

\newtheorem{thm}{Theorem}[section]

\newtheorem{prop}[thm]{Proposition}

\theoremstyle{definition}

\theoremstyle{remark}
\newtheorem{rem}{Remark}[section]

\newcommand{\be}{\begin{equation}}
\newcommand{\ee}{\end{equation}}
\newcommand{\bea}{\begin{eqnarray}}
\newcommand{\eea}{\end{eqnarray}}
\newcommand{\ben}{\begin{eqnarray*}}
	\newcommand{\een}{\end{eqnarray*}}
\newcommand{\bt}{\begin{split}}
	\newcommand{\et}{\end{split}}
\newcommand{\bet}{\begin{equation}}

\begin{document}
\title[Extension of quasi-plurisubharmonic functions]{On the extension of K\"ahler currents on compact K\"ahler manifolds: holomorphic retraction case}

\author[J. Ning]{Jiafu Ning}
\address{Jiafu Ning: \ Department of Mathematics, Central South University, Changsha, Hunan 410083, P. R. China.}
\email{jfning@csu.edu.cn}
\author[Z. Wang]{Zhiwei Wang}
\address{ Laboratory of Mathematics and Complex Systems (Ministry of Education)\\ School of Mathematical Sciences\\ Beijing Normal University\\ Beijing 100875\\ P. R. China}
\email{zhiwei@bnu.edu.cn}
\author[X. Zhou]{Xiangyu Zhou}
\address{Xiangyu Zhou: Institute of Mathematics\\Academy of Mathematics and Systems Sciences\\and Hua Loo-Keng Key
	Laboratory of Mathematics\\Chinese Academy of
	Sciences\\Beijing\\100190\\P. R. China}
\address{School of
	Mathematical Sciences, University of Chinese Academy of Sciences,
	Beijing 100049, China}
\email{xyzhou@math.ac.cn}

\begin{abstract}
In the present paper, we show that given a compact K\"ahler manifold $(X,\omega)$ with a K\"ahler metric $\omega$, and a complex submanifold $V\subset X$ of positive dimension, if $V$ has a holomorphic retraction structure in $X$, then any quasi-plurisubharmonic function $\varphi$ on $V$ such that $\omega|_V+\sqrt{-1}\partial\bar\partial\varphi\geq \varepsilon\omega|_V$  with $\varepsilon>0$ can be extended to a quasi-plurisubharmonic function $\Phi$ on $X$, such that $\omega+\sqrt{-1}\partial\bar\partial \Phi\geq \varepsilon'\omega$ for some $\varepsilon'>0$. This is an improvement of results in \cite{WZ20}. Examples satisfying the assumption that there exists a holomorphic retraction structure contain product manifolds, thus contains many compact K\"ahler manifolds which are not necessarily projective. 
\end{abstract}

\thanks{The first author was partially supported by the NSFC grant(12071485). 
	The second author was partially supported by the Beijing Natural Science Foundation (1202012, Z190003) and by the NSFC grant (11701031,12071035). The third author was partially supported by the NSFC grant (11688101). }

\maketitle

\section{Introduction}
In this paper, we study  the following problem raised by Coman-Guedj-Zeriahi.
\begin{prob}[\cite{CGZ13}]\label{prob:1}
Let $(X,\omega)$ be a compact K\"ahler manifold of complex dimension $n$, equipped with a K\"ahler metric $\omega$. Let $V\subset X$ be a complex submanifold of complex dimension $k>0$. Is the following holds
\begin{align*}
\mbox{Psh}(V,\omega|_V)=\mbox{Psh}(X,\omega)|_V?
\end{align*}
\end{prob}

Recently, there are many progress towards this problem. 
\begin{itemize}
\item When $\omega$ is a Hodge metric and $\varphi$ is a smooth quasi-psh function on $V$, such that $\omega|_V+\sqrt{-1}\partial\bar\partial \varphi>0$, then Problem \ref{prob:1} has a positive answer by Schumacher \cite{Sch98}.
\item When $\omega$ is a Hodge metric, then Problem \ref{prob:1} has a positive answer by Coman-Guedj-Zeriahi \cite{CGZ13}, and when $\omega$ is a K\"ahler metric and $\varphi$ is a smooth quasi-psh function on $V$, such that $\omega|_V+\sqrt{-1}\partial\bar\partial \varphi>0$, then Problem \ref{prob:1} has a positive answer in the same paper \cite{CGZ13}.
\item When  $\omega$ is a K\"ahler metric and $\varphi$ is a  quasi-psh function  on $V$, which has analytic singularities, such that $\omega|_V+\sqrt{-1}\partial\bar\partial \varphi>\epsilon\omega|_V$ for some $\epsilon>0$, there is a quasi-psh function $\Phi$ on $X$, such that $\Phi|_V=\varphi$ and $\omega+\sqrt{-1}\partial\bar\partial \Phi>\epsilon'\omega$ on $X$ by Collins-Tosatti \cite{CT14}.
\item When  $\omega$ is a K\"ahler metric and $\varphi$ is a  quasi-psh function with arbitrary singularity  on $V$, such that $\omega|_V+\sqrt{-1}\partial\bar\partial \varphi>\epsilon\omega|_V$ for some $\epsilon>0$. Suppose that $V$ has a holomorphic tubular neighborhood in $X$, then there is a quasi-psh function $\Phi$ on $X$, such that $\Phi|_V=\varphi$ and $\omega+\sqrt{-1}\partial\bar\partial \Phi>\epsilon'\omega$ on $X$ for some $\epsilon'>0$ by Wang-Zhou \cite{WZ20}.
\end{itemize}

The main theorem of this paper is as follows.
\begin{thm}\label{thm:main}
Let $(X,\omega)$ be a compact K\"ahler manifold of complex dimension $n$, equipped with a K\"ahler metric $\omega$. Let $V\subset X$ be a complex submanifold of complex dimension $k>0$. Let $\varphi$ be  a  quasi-psh function with arbitrary singularity  on $V$, such that $\omega|_V+\sqrt{-1}\partial\bar\partial \varphi>\epsilon\omega|_V$ for some $\epsilon>0$.  Suppose that there is a open neighborhood $U$ of $V$ in $X$, and  a holomorphic retraction $\pi:U\rightarrow V$. Then  there is a quasi-psh function $\Phi$ on $X$, such that $\Phi|_V=\varphi$ and $\omega+\sqrt{-1}\partial\bar\partial \Phi>\epsilon'\omega$ on $X$  for some $\epsilon'>0$.
\end{thm}

\begin{rem}
The main theorem is a slightly stronger than the result in \cite{WZ20}, by weakening the assumption that $V$ has a holomorphic tubular neighborhood structure in $X$ to the assumption that $V$ has a holomorphic retraction structure in $X$. By a holomorphic retraction, we mean that there is an  open neighborhood $U$ of $V$ in $X$, and  a holomorphic map $\pi:U\rightarrow V$, such that $\pi|_V:V\rightarrow V$ is the identity map.  Without the holomorphic tubular neighborhood structure,  we need to compute the complex Hessian of the square of the distance function to $V$ on $X$. The main advantage of this generalization is that there many examples of compact  K\"ahler  manifolds  which are  not necessarily projective.
\end{rem}
We also consider the extension of K\"ahler currents in a big class.
\begin{thm}
Let $(X,\omega)$ be a compact K\"ahler manifold of complex dimension $n$, and $V\subset X$ be a complex submanifold of positive dimension. Suppose that $V$ has a holomorphic retraction structure  in $X$. 
Let $\alpha\in H^{1,1}(X,\mathbb R)$ be a big class and $E_{nK}(\alpha)\subset V$. 
Then any  K\"ahler current in  $\alpha|_V$ is the restriction of a K\"ahler current in $\alpha$.
\end{thm}

\subsection*{Acknowledgement}
The second author would like to thank G. Hosono and T. Koike for helpful discussions, and especially T. Koike for  sharing the note \cite{Ko21}.

\section{Complex Hessian of square of  distance to  a complex  submanifold}\label{sect:dist}
We follow  Matsumoto's notations in \cite{Ma}.
Let $(M,g)$ be a $C^\infty$ Rimannian manifold of dimension $n$.
For $x,y\in M$, we denote by $\delta(x,,y)$ the distance between $x$ and $y$ induced by the metric $g$.

It is known that for any $p\in M$, there is an open coordinate neighborhood $U$ of $p$,
choose a coordinate $x_1,x_2,\cdots,x_n$ on $U$, with $x(p)=0$ and
$g_{ij}=g(\frac{\partial}{\partial x_i},\frac{\partial}{\partial x_j}),\, i,j=1,2,\cdots,n$.
For $v=(v_1,v_2,\cdots,v_n)\in\mathbb{R}^n$,
we may view $v\in T_xM$ as $\sum_{j=1}^{n}v_j\frac{\partial}{\partial x_j}|_x$.
We may shrink $U$ if necessary, there is an open neighborhood $B\subset \mathbb{R}^n$ of $0$,
such that $\Phi(x,v)=(x,\exp_x(v))$
is bijection from $U\times B$ to $\Phi(U\times B)$, both $\Phi$ and $\Phi^{-1}$ are $C^\infty$.
As $\Phi(x,0)=(x,x)$, and from the property of exponential mapping $y=\exp_x v$,
we can get
\begin{equation}\label{eq1}
J\Phi(0,0)
=\left[
\begin{array}{cc}
I & 0\\
I & I
\end{array}
\right].
\end{equation}

As $\Phi(U\times B)$ is an open neighborhood of $(p,p)$,
we may take an open set $V\subset U$, such that $p\in V$, and  $\Phi(U\times B)\supset V\times V$.
Write $(x,v(x,y))=\Phi^{-1}(x,y)$, then
$$y=\exp_x(v(x,y)),\, \delta(x,y)^2=\sum_{i,j=1}^n g_{ij}(x)v_i(x,y)v_j(x,y)$$
and from (\ref{eq1}), we have
\begin{equation}\label{eq2}
v(0,0)=0, \,\, \frac{\partial v_i}{\partial y_j}(0,0)=-\frac{\partial v_i}{\partial x_j}(0,0)=\delta_{ij},\, 1\leq i,j\leq n.
\end{equation}
Let $S\in M$ be a $C^\infty$ submanifold of $M$ with $dim S=k$, $0\leq r<n$, define
$$d(x)=\delta(x,S)=inf\{\delta(x,y):y\in S\}, \; x\in M.$$
Let $h(x)=d(x)^2$, it is  proved in \cite{Ma} that $h(x)$ is $C^\infty$ in a tube neighborhood of $S$.
As we need do calculation on $h$, and for completeness, we introduce Matsumoto's proof of smoothness.

Given $p\in S$, we can choose a coordinate $(U,x=(x_1,x_2,\cdots,x_n))$ around $p$, such that
$x(p)=0$, $S\cap U=\{x_{k+1}=\cdots=x_n=0\}$, and
\begin{equation}\label{eq3}
g_{ij}(0)=g(\frac{\partial}{\partial x_i},\frac{\partial}{\partial x_j})(0)=\delta_{ij},\, i,j=1,2,\cdots,n.
\end{equation}
Take a small neighborhood $V\subset U$ of $p$, such that $d(x)=\delta(x,S\cap U)$ for any $x\in V$.
Let $$f(x,y)=\delta(x,y)^2=\sum_{i,j=1}^{n}g_{ij}(x)v_i(x,y)v_j(x,y)$$
and $$F_\mu(x,y)=\frac{\partial f}{\partial y_\mu}(x,y),\; \mu=1,2,\cdots,k$$
for $x\in V$ and $y=(y_1,\cdots,y_k,0,\cdots,0)\in S\cap U$. Notice that $F_\mu(0,0)=0, \mu=1,2,\cdots,k$.

From (\ref{eq2}), (\ref{eq3}), we can get
\begin{equation}\label{eq4}
\begin{split}
&\frac{\partial F_\mu}{\partial y_\nu}(0,0)\\
=&2\sum_{i,j=1}^{n}g_{ij}(0)\frac{\partial v_i}{\partial y_\mu}(0,0)\frac{\partial v_j}{\partial y_\nu}(0,0)\\
=&2\delta_{\mu\nu}.
\end{split}
\end{equation}

Therefore, from the implicit function theorem, we can find a neighborhood $V_0\subset V$ of $p$,
so that each $x\in V_0$ has a unique solution $y=y(x)\in S\cap U$ of equations $F_\mu(x,y)=0, \mu=1,2,\cdots,n$,
$y(0)=0$ and $y=y(x)$ is $C^\infty$ on $V_0$. As for each $x\in V_0$, there is at least one point $y\in S\cap U$, 
such that $d(x)=\delta(x,S\cap U)=\delta(x,y)$. Therefore, the point $y$ is uniquely determined by $x$ and 
it must coincide to $y(x)$ because $f(x,y)=\delta(x,y)^2$ is minimal at $y=y(x)$ for each $x$.

Hence $$h(x)=\delta(x,S\cap U)^2=\delta(x,y(x))^2=f(x,y(x))$$
for $x\in V_0$.

From $F_\mu(x,y(x))=0$, taking the partial derivatives of this equation, we can get 
$$\frac{\partial F_\mu}{\partial x_j}(0,0)+\sum_{\nu=1}^{k}\frac{\partial F_\mu}{\partial y_\nu}(0,0)
\frac{\partial y_\nu}{\partial x_j}(0)=0.$$ 
Combining with  equation (\ref{eq4}),
we can get $\frac{\partial y_\mu}{\partial x_j}(0)=-\frac{1}{2}\frac{\partial F_\mu}{\partial x_j}(0,0)$
for $\mu=1,2,\cdots,k$ and $j=1,2,\cdots,n$.

From (\ref{eq2}), (\ref{eq3}), we have

\begin{equation*}
\begin{split}
&\frac{\partial F_\mu}{\partial x_l}(0,0)\\
=&\sum_{i,j=1}^{n}g_{ij}(0)\frac{\partial v_i}{\partial y_\mu}(0,0)\frac{\partial v_j}{\partial x_l}(0,0)+
\sum_{i,j=1}^{n}g_{ij}(0)\frac{\partial v_i}{\partial x_l}(0,0)\frac{\partial v_j}{\partial y_\mu}(0,0)\\
=&-2\delta_{\mu l}.
\end{split}
\end{equation*}

Hence, 
\begin{equation}\label{eq d}
\frac{\partial y_\mu}{\partial x_j}(0)=\delta_{\mu j}
\; \text{for}\; \mu=1,2,\cdots,k  ; j=1,2,\cdots,n.
\end{equation}

Let $a_j(x)=v(x,y(x)), j=1,2,\cdots,n$, then $a_j(0)=0$, and 
\begin{equation}\label{eq6}
\begin{split}
\frac{\partial a_j}{\partial x_l}(0)&=\frac{\partial v_j}{\partial x_l}(0,0)+
\sum_{\mu=1}^{k}\frac{\partial v_j}{\partial y_\mu}(0,0)\frac{\partial y_\mu}{\partial x_k}(0)\\
&=\sum_{\mu=1}^{k}\delta_{j\mu}\delta_{\mu k}-\delta_{jk}.
\end{split}
\end{equation}
As $h(x)=\sum_{i,j=1}^{n}g_{ij}(x)v_i(x,y(x))v_j(x,y(x))=\sum_{i,j=1}^{n}g_{ij}(x)a_i(x)a_j(x)$, from (\ref{eq6}), then
\begin{equation}\label{eq 7}
\begin{split}
\frac{\partial^2h}{\partial x_s\partial x_t}(0)
&=2\sum_{i,j=1}^{n}g_{ij}(0)\frac{\partial a_i}{\partial x_s}(0)\frac{\partial a_j}{\partial x_t}(0)\\
&=
\begin{cases}
0 \quad\quad s \; \text{or} \; t\leq k \\
2\delta_{st} \quad s,t>k
\end{cases}
\end{split}
\end{equation}

Now we let $(X,\omega)$ be a compact Hermitian manifold with  a Hermitian metric $\omega$. Let $g$ be the Riemannian metric on $X$ induced by $\omega$. Let $V\subset X$ be a complex submanifold of complex dimension $r>0$.  Fix any  $p\in V$. There is a holomorphic coordinate  $(U, z=(z_1,\cdots, z_k, z_{k+1}, \cdots,z_n))$ centered at $p$ in $X$, such that $U\cap V=\{z_{k+1}=\cdots=z_n=0\}$, and $g_{ij}(0)=\delta_{ij}$ for $i,j=1,\cdots, 2n$, here we write $z_i=x_{2i-1}+\sqrt{-1}x_{2i}$. Since  $$\frac{\partial}{\partial z_i}=\frac{1}{2}(\frac{\partial}{\partial x_{2i-1}}-\sqrt{-1}\frac{\partial}{\partial x_{2i}}), \frac{\partial}{\partial\bar z_i}=\frac{1}{2}(\frac{\partial}{\partial x_{2i-1}}+\sqrt{-1}\frac{\partial}{\partial x_{2i}}),$$
we get that 
\begin{align}\label{equ: chessian}
\frac{\partial^2h}{\partial z_i\partial \bar z_j}=\frac{1}{4}\left(\frac {\partial^2h}{\partial x_{2i-1}\partial x_{2j-1}}       +\frac{\partial^2h}{\partial x_{2i}\partial x_{2j}}-\sqrt{-1}\frac{\partial^2h}{\partial x_{2j}\partial x_{2j-1}} +\sqrt{-1}\frac{\partial^2h}{\partial x_{2j-1}\partial x_{2j}}    \right).
\end{align}

Combining (\ref{eq 7}) and (\ref{equ: chessian}), we obtain the following 

\begin{prop}\label{prop:hessian}Let $(X,\omega)$ be a complex $n$-dimensional  Hermitian manifold with a Hermitian metric $\omega$. Let $V\subset X$ be a complex submanifold of complex dimension $k$. Let $p\in V$ be an arbitrarily fixed point in $V$, then there is a holomorphic coordinate chart  $(U,z=(z_1,\cdots,z_k,z_{k+1},\cdots, z_n))$ centered at $p$ such that $U\cap V=\{z_{k+1}=\cdots=z_n=0\}$ and 
	\begin{align*}
	\frac{\partial^2h}{\partial z_i\partial \bar z_j}(0)=
	\begin{cases}
	0 \quad\quad i \; \text{or} \; j\leq k \\
	2\delta_{ij} \quad i,j>k.
	\end{cases}
	\end{align*}
	\end{prop}

\section{Proof of the main theorem}

In this section, we give the proof of Theorem \ref{thm:main}. The idea of the proof is similar with that in \cite{WZ20}. The main difference lies in the construction of the local uniform extension. For the sake of completeness, we give the detailed proof.

\begin{lemma}[\cite{BK07, WZ20}]\label{key lemma}
	Let $\varphi$ be a quasi-psh function on a compact Hermitian  manifold $(X,\omega)$,  such that $\omega+\sqrt{-1}\partial\bar{\partial}\varphi\geq \varepsilon\omega$ and $\varphi<C<0$.
	Then there is a  sequence of smooth functions $\varphi_m$ and a decreasing sequence $\varepsilon_m>0$ converging to $0$, satisfying the following
	\begin{itemize}
		\item [(a)] $\varphi_m\searrow \varphi$;
		\item [(b)]$\omega+\sqrt{-1}\partial\bar{\partial}\varphi_m\geq (\varepsilon-\varepsilon_m)\omega$;
		\item[(c)] $\varphi_m\leq -\frac{C}{2}$.
	\end{itemize}
\end{lemma}

\begin{lemma}[c.f. \cite{DP04}]\label{reference function}
	There exists a function $F:X\rightarrow [-\infty, +\infty)$ which is smooth on $X\setminus V$, with logarithmic singularities along $V$, and such that $\omega+\sqrt{-1}\partial\bar{\partial}F\geq \varepsilon \omega$ is  a K\"{a}hler current on $X$.
	By subtracting a large constant, we can make that $F<0$ on $X$.
\end{lemma}

Let $T=\omega|_V+\sqrt{-1}\partial\bar{\partial}\varphi\geq \varepsilon\omega|_V$ be the given K\"{a}hler current  in the K\"{a}hler class $[\omega|_V]$, where $\varphi$ is a strictly $\omega|_V$-psh function.
By subtracting a large constant, we may assume that $\sup_V \varphi<-C$ for some positive constant $C$.

By  Lemma \ref{key lemma}, we have that there is a sequence of non-increasing  smooth strictly $\omega|_V$-psh functions $\varphi_{m}$ on $V$,
and a decreasing sequence of positive numbers $\varepsilon_m$ such that as $m\rightarrow \infty$
\begin{itemize}

	\item $\varphi_{m} \searrow \varphi$;
	\item$\omega|_V+\sqrt{-1}\partial\bar{\partial}\varphi_m> \frac{\varepsilon}{2}\omega|_V$;
	\item $\varphi_m\leq -\frac{C}{2}$.
\end{itemize}

We say a smooth strictly $\omega|_V$-psh function $\phi$ on $V$ satisfies \textbf{assumption $\bigstar_{\varepsilon, C}$}, if $\omega|_V+\sqrt{-1}\partial\bar{\partial}\phi>\frac{\varepsilon}{2}\omega|_V$ and $\phi<-\frac{C}{2}$.

Note that for all $m\in \mathbb{N}^+$, $\varphi_m$ satisfy  \textbf{assumption $\bigstar_{\varepsilon, C}$}.  In the following, we will extend all the $\varphi_m$ simultaneously to non-increasing  strictly $\omega$-psh functions on   the ambient manifold $X$.

\noindent\textbf{Step1: Local uniform extensions of $\varphi_m$ for all $m$.} Let $U\subset X$ be an open neighborhood of $V$ and let $r:U\rightarrow V$ be a holomorphic retraction. Let $\phi$ be a function satisfying  \textbf{assumption $\bigstar_{\varepsilon, C}$}.  Let $h$ be the square of the distance function , which is a smooth function defined in \S\ref{sect:dist}.
 We define 	
\begin{align*}
\bar{\phi}:=\phi\circ r+Ah
\end{align*}
where $A$ is a positive constant to be determined later. 

Fix arbitrary $p\in V$, choose a holomorphic coordinate chart $(W_p,z=(z_1,\cdots,z_n))$ centered at $p$ and $W_p\cap V=\{z_{k+1}=\cdots=z_n=0\}$, $g_{i\bar j}(0)=\delta_{ij}$  and $W\subset U$, where $\omega=\sum_{i,j=1}^ng_{i\bar j}dz_i\wedge d\bar z_j$. Then on $W_p$, we have that 
\begin{align*}
\bar{\phi}(z):=(\phi\circ r)(z)+Ah(z).
\end{align*}
Note that on $W_p$, 
\begin{align*}
\omega+\sqrt{-1}\partial\bar\partial \bar\phi(z)&=(\omega-r^*(\omega|_V))+r^*(\omega|_V+\sqrt{-1}\partial\bar\partial\phi)+A\sqrt{-1}\partial\bar\partial h\\
&\geq (\omega-r^*(\omega|_V))+\varepsilon r^*(\omega|_V)+A\sqrt{-1}\partial\bar\partial h
\end{align*}
The second inequality follows from the fact that $\omega|_V+\sqrt{-1}\partial\bar\partial\phi\geq \varepsilon\omega|_V$  on $V$ and $r$ is a holomorphic retraction map.  The key point is that the last term in above inequality  is independent of $\phi$.


\begin{claim}\label{claim: 1}There is  an open neighborhood  $W_p$ (independent of $\phi$), of $p$ in $U$, and  positive constants $A>0$ and  $\varepsilon'>0$ (independent of $\phi$), such that on $W_p$,
	
\begin{align*}
\omega+\sqrt{-1}\partial\bar\partial \bar\phi\geq \frac{\varepsilon'}{2}\omega \mbox{~~and~~} \bar\phi\leq -\frac{C}{4}.
\end{align*}
\end{claim}
\begin{proof}Under the local coordinate chosen as above, one can see that 
	\begin{align*}
	&r(z)=(r_1(z_1,\cdots,z_n),\cdots, r_k(z_1,\cdots,z_n),0,\cdots,0);\\
&	r(z_1,\cdots, z_k,0,\cdots, 0)=(z_1,\cdots,z_k,0,\cdots,0);\\
&dr_i(z_1,\cdots,z_k,0,\cdots,0)=dz_i+\sum_{k+1\leq j\leq n}\frac{\partial r_i}{\partial z_j}(z_1,\cdots, z_k,0\cdots,0)dz_j.
	\end{align*}
Since $\omega|_V=\sum_{1\leq i,j\leq k}g_{i\bar j}(z_1,\cdots,z_k,0,\cdots,0)dz_i\wedge d\bar z_j$, it follows that at $(z_1,\cdots, z_k,0,\cdots,0)$, 
\begin{align*}
r^*(\omega|_V)=&\sum_{1\leq i,j\leq k}g_{i\bar j}(dz_i+\sum_{k+1\leq l\leq n}\frac{\partial r_i}{\partial z_l}dz_l)\wedge(d\bar z_j+\sum_{k+1\leq m\leq n}\frac{\partial \bar r_j}{\partial \bar z_m} d\bar z_m)\\
=&\sum_{1\leq i,j\leq k}g_{i\bar j}dz_i\wedge d\bar z_j+\sum_{1\leq i\leq k,k+1\leq m\leq n}\sum_{1\leq j\leq k}g_{i\bar j}\frac{\partial \bar r_j}{\partial \bar z_m}dz_i\wedge d\bar z_m\\
&+\sum_{1\leq j\leq k,k+1\leq l\leq n}\sum_{1\leq i\leq k}g_{i\bar j}\frac{\partial  r_i}{\partial \bar z_l}dz_i\wedge d\bar z_l
+\sum_{k+1\leq l,m\leq n}\sum_{1\leq i,j\leq k}g_{i\bar j}\frac{\partial  r_i}{\partial  z_l}\frac{\partial  \bar r_j}{\partial \bar z_m}dz_l\wedge d\bar z_m.
\end{align*}

Thus,  at $(z_1,\cdots,z_k,0,\cdots,0)$, we get the following
\begin{align*}
\omega+\sqrt{-1}\partial\bar\partial \bar\phi(z)\geq &\sum_{1\leq i,j\leq k}(\varepsilon g_{i\bar j}+Ah_{i\bar j})dz_i\wedge d\bar z_j+\sum_{1\leq i\leq k,k+1\leq m\leq n}(g_{i\bar m}+Ah_{i\bar m}+(\varepsilon-1)\sum_{1\leq j\leq k}g_{i\bar j}\frac{\partial \bar r_j}{\partial \bar z_m})dz_i\wedge d\bar z_m\\
&+\sum_{1\leq j\leq k,k+1\leq l\leq n}(g_{j\bar l}+Ah_{j\bar l}+(\varepsilon-1)\sum_{1\leq i\leq k}g_{i\bar j}\frac{\partial  r_i}{\partial \bar z_l})dz_i\wedge d\bar z_l\\
&+\sum_{k+1\leq i,j\leq n}(g_{i\bar j}+Ah_{i\bar j}+(\varepsilon-1)\sum_{1\leq i,j\leq k}g_{i\bar j}\frac{\partial  r_i}{\partial  z_l}\frac{\partial  \bar r_j}{\partial \bar z_m})dz_i\wedge d\bar z_j.
\end{align*}
Since $(g_{i\bar j})_{1\leq i,j\leq k}$ is positive definite, from Proposition \ref{prop:hessian}, we can see that when $A>0$ is sufficiently large (independent of $\phi$), there is an open neighborhood $W_p$ (independent of $\phi$), such that the conclusion of Claim \ref{claim: 1} holds.

	\end{proof}

To emphasis the uniformity, it is worth to point out again that the chosen of the open set $W_p$, and  the constant $\varepsilon'$ is independent of $\phi$, as long as $\phi$ satisfies \textbf{assumption $\bigstar_{\varepsilon,C}$}.
We call the above data $ (W_p,\varepsilon',-\frac{C}{4},\bar\phi)$ an \textbf{admissible local extension} of $\phi$.

Since all the $\varphi_m$ satisfy the same \textbf{assumption $\bigstar_{\varepsilon,C}$}, thus near $p$, we can choose a \textbf{uniform admissible local extension $ (W_p,A,\varepsilon',-\frac{C}{4},\bar\varphi_m)$} of $\varphi_m$, for all $m\in\mathbb{N}^+$.
Since $V$ is compact, one may choose an  open neighborhood $W$ of $V$ in $X$, and universal constants $A>0$ and $\varepsilon'>0$, such that the functions $\widetilde \varphi_m:=\varphi_m\circ r+Ah$ are defined on $W$, such that $\omega+i\partial\bar\partial \widetilde \varphi_m\geq \varepsilon'\omega$ on $W$ for all $m$. Since $\{\varphi_m\}$ is a non-increasing sequence, one obtains that $\{\widetilde{\varphi}_m\}$ is a non-increasing sequence.

\noindent\textbf{Step 2: Global extensions of $\varphi_m$ for all $m$.} Up to shrinking, we may assume that $\widetilde{\varphi}_m$ are defined on the closure of $W$ for all $m\in \mathbb{N}^+$.
Let $F$  be the quasi-psh function in Lemma \ref{reference function}.
Near $\partial W$ (the boundary of $W$), the function $F$ is smooth, and $\sup_{\partial W}\widetilde{\varphi}_{1}=-C''$ for some positive constant $C''>0$.
Now we choose a small positive $\nu$, such that $\inf_{\partial W}(\nu F)>-\frac{C''}{2}$ and   $\omega+i\partial\bar{\partial}\nu F\geq\varepsilon'\omega$.
Thus $\nu F >\widetilde{\varphi}_{1}\geq \widetilde{\varphi}_m$ in a neighborhood of $\partial W$ for all $m\in \mathbb{N}^+$, since $\widetilde{\varphi}_m$ is non-increasing.
Therefore, we can finally define
\begin{align*}
\Phi_m=\left\{
\begin{array}{ll}
\max\{\widetilde{\varphi}_m, \nu F\}, & \hbox{on $W$;} \\
\nu F, & \hbox{on $X\setminus W$,}
\end{array}
\right.
\end{align*}
which is defined on the whole of $X$. It  is easy to check that $\Phi_m$ satisfies the following properties:
\begin{itemize}
	\item $\Phi_m$ is non-increasing in  $m$,
	\item $\Phi_m\leq 0$ for all $m\in \mathbb{N}^+$,
	\item $\omega+i\partial\bar{\partial}\Phi_m\geq \varepsilon'\omega$ for all $m\in \mathbb{N}^+$,
	\item $\Phi_m|_V=\varphi_m$ for all $m\in \mathbb{N}^+$.
\end{itemize}

\noindent\textbf{Step 3: Taking limit to complete the proof of Theorem \ref{thm:main}.}
From above steps, we get a sequence of non-increasing, non-positive strictly $\omega$-psh functions $\Phi_m$ on $X$. Then  either $\Phi_m\rightarrow -\infty $ uniformly on $X$, or $\Phi:=\lim\limits_m\Phi_m\in$ Psh$(X,\omega)$.
But $\Phi_m|_V=\varphi_m\searrow \varphi\not\equiv -\infty$, the first case will not appear.
Moreover,   we can see that $\Phi:=\lim\limits_m\Phi_m$ is a strictly $\omega$-psh function on $X$ from the property $\omega+i\partial\bar{\partial}\Phi_m\geq \varepsilon'\omega$ for all $m\in \mathbb{N}^+$, and $\Phi|_V=\lim\limits_m\Phi_m|_V=\lim\limits_m\varphi_m=\varphi$.
It follows that   $(\omega+i\partial\bar{\partial}\Phi)|_V=\omega|_V+i\partial\bar{\partial}\varphi$.
Thus we complete the proof of Theorem \ref{thm:main}.

\begin{rem} By similar arguments as in \cite{WZ20}, we can get the following extension results for K\"ahler currents in a big class.
	\begin{thm}
		Let $(X,\omega)$ be a compact K\"ahler manifold of complex dimension $n$, and $V\subset X$ be a complex submanifold of positive dimension. Suppose that $V$ has a holomorphic retraction structure  in $X$. 
		Let $\alpha\in H^{1,1}(X,\mathbb R)$ be a big class and any of the irreducible components of $E_{nK}(\alpha)$ either does not intersect with $V$, or is contained in $V$. 
		Then any  K\"ahler current in  $\alpha|_V$ is the restriction of a K\"ahler current in $\alpha$.
	\end{thm}
	\end{rem}

\section{Examples}

In \cite{HK20}, Hosono-Koike point out that in Nakayama's example and Zariski's example, the submanifold  have holomorphic tubular neighborhood structure in  the ambient manifold, thus have  holomorphic retraction structure. 

\noindent{\textbf{Product manifold.}} Let $Y_1$ and $Y_2$ be two compact K\"ahler manifold and set $X:=Y_1\times Y_2$. Fix an arbitrary point $p\in Y_2$, let $V=Y_1\times p$, then the natural map $\pi:Y_1\times Y_2\rightarrow Y_1\times p$  serves as a holomorphic retraction map.

An interesting  example of non-product manifold, communicated to us by Koike \cite{Ko21}, is  the following famous example of Serre.

\noindent{\textbf{Serre's example.}} Let $X:=\mathbb P_{[x;y]}\times \mathbb C_z/\sim$, where $\tau\in \mathbb H$ with $\mathbb H$ be the upper half plane, and $$([x;y],z)\sim ([x;y+x],z+1)\sim ([x;y+\bar\tau\cdot x],z+\tau).$$
Let $V:=\{x=0\}\subset X$ as a submanifold of $X$ which is obviously isomorphic to the elliptic curve $\mathbb C/\langle 1,\tau\rangle$.  It is easy to check  that  the projection map $\pi: X\rightarrow \mathbb C/\langle 1,\tau\rangle=:V$ is a holomorphic retraction. It can also be verified that $V$ does not have a holomorphic tubular neighborhood structure in $X$.
\begin{rem}
	In \cite{Ko21}, Koike gives a very interesting  proof of Theorem \ref{thm:main} for Serre's example,  which however seems not applicable to general case treated in this paper.
\end{rem}

	\end{document}